\documentclass{svjour3}                     % onecolumn (standard format)
\smartqed  % flush right qed marks, e.g. at end of propositionof
\usepackage{graphicx}
\usepackage{amssymb,amsfonts}
\usepackage{amsmath}
\usepackage{algorithm}
\usepackage{algorithmic}
\usepackage{latexsym}
\numberwithin{theorem}{section}
\numberwithin{definition}{section}
\numberwithin{proposition}{section}
\numberwithin{lemma}{section}
\numberwithin{remark}{section}
\numberwithin{example}{section}
\numberwithin{algorithm}{section}

\begin{document}

\title{Existence of Hukuhara differentiability of fuzzy-valued functions}

\titlerunning{}        % if too long for running head
\author{U. M. Pirzada         \and
        D. C. Vakaskar %etc.
}

\institute{U. M. Pirzada \and D. C. Vakaskar \at
              Department of Applied Mathematics, Faculty of Tech. {\&} Engg., M.S.University of Baroda, Vadodara-390001, India. \\
              \email{dcvakaskar@gmail.com}           %  \\
           \and
           U. M. Pirzada \\
	      \email{salmapirzada@yahoo.com}  
}

%\author{Umme Salma M. Pirzada }

%%\authorrunning{U. M. Pirzada} % if too long for running head

%\institute{Umme Salma M. Pirzada \at
%              Department of Applied Mathematics, Faculty of Tech. {\&} Engg., M.S.University of Baroda, Vadodara-390001, India. \\
%	      \email{salmapirzada@yahoo.com}  
%}

\date{}

\maketitle

\begin{abstract}
In this paper, we discuss existence of Hukuhara differentiability of fuzzy-valued functions. Several examples are worked out to check that fuzzy-valued functions are one time, two times and n-times H-differentiable. We study the effects of fuzzy modelling on existence of Hukuhara differentiability of fuzzy-valued functions. We discuss existence of gH-differentiability and its comparison with H-differentiability.
\end{abstract}

\section{Introduction}
\label{intro}

\begin{par}
Hukuhara differentiability of fuzzy-valued functions is generalization of Hukuhara differentiability of set-valued functions. This differentiability is based on Hukuhara difference. Hukuhara introduced this difference (subtraction) of two sets in \cite{HU}. He introduced the notions of integral and derivative for set-valued mappings and considered the relationship between them. This derivative is widely studied and analysed by researchers for set-valued as well as fuzzy-valued functions. Wide range of applications of Hukuhara derivatives are studied in fuzzy differential equations and fuzzy optimization problems. Unfortunately, the derivative is very restrictive. Its existence based on certain conditions. We found that it is very interesting to study and analyse these conditions for several basic and important fuzzy-functions. This study helps the researchers to identify the class of fuzzy-functions which are Hukuhara differentiable and can be utilized in applications like fuzzy optimization, fuzzy ordinary and partial differential equations. We explore fuzzy modelling part which is crucial in any fuzzy theory and study its effect on existence of Hukuhara differentiability of fuzzy-valued function by giving appropriate examples. 
\end{par}
\begin{par}
An improvement of Hukuhara difference is generalized Hukuhara difference ( gH-difference) defined in \cite{ST1,ST2}. 
It is more general than Hukuhara difference (H-difference). Bede and Stefanini \cite{BE1} have defined generalized Hukuhara differentiability of fuzzy-valued functions based on gH-difference. We discuss existence of this differentiability and its comparison with H-differentiability.
\end{par}
The paper is organized in following manner. Section 2 contains basics of fuzzy numbers. H-differentiability of a fuzzy-valued function is discussed and compared with gH-differentiability with several examples in Section 3. Effect of modelling on existence of H-differentiability is discussed in Section 4. Conclusion is given in the last section.
%======================================================================================================================
\section{Fuzzy numbers and arithmetic}
We start with some basic definitions.
\label{sec:2}
\begin{definition}\label{def1}\cite{GE} Let $\mathbb{R}$ be the set of real numbers and $\tilde{a}:\mathbb{R} \to[0,1]$ be a fuzzy set. We say that $\tilde{a}$ is a fuzzy number if it satisfies the following properties:
\begin{description}
        \item[(i)] {$\tilde{a}$ is normal, that is, there exists $x_0 \in \mathbb{R} $ such that $\tilde{a}(x_0)=1$;}
        \item[(ii)] {$\tilde{a}$ is fuzzy convex, that is,  $\tilde{a}(t x+(1-t) y)\geq min \{\tilde{a}(x),\tilde{a}(y)\}$, whenever x, y $\in
\mathbb{R}~~ and~~~t \in [0,1]$;}
        \item[(iii)] {$\tilde{a}(x)$ is upper semi-continuous on $\mathbb{R}$, that is, $\{x / \tilde{a}(x) \geq \alpha \}$ is a closed subset of $\mathbb{R}$ for each $\alpha \in (0,1]$;}
        \item[(iv)] {$cl\{x \in \mathbb{R} / \tilde{a}(x) >0\}$ forms a compact set,}
    \end{description}

where $cl$ denotes closure of a set. The set of all fuzzy numbers on $\mathbb{R}$ is denoted by $F(\mathbb{R})$. For all $\alpha \in (0,1]$, $\alpha$-level set $\tilde{a}_{\alpha}$ of any $\tilde{a}\in F(\mathbb{R})$ is defined as $\tilde{a}_{\alpha} = \{ x \in \mathbb{R}/ \tilde{a}(x)\geq \alpha \}$ . The 0-level set $\tilde{a}_{0}$ is defined as the closure of the set $\{x \in \mathbb{R} / \tilde{a}(x) >0\}$. By definition of fuzzy numbers, we can prove that, for any $\tilde{a}\in F(\mathbb{R})$ and for each $\alpha \in (0,1]$ , $\tilde{a}_{\alpha}$ is compact convex subset of $\mathbb{R}$, and we write $\tilde{a}_{\alpha}=[{\tilde{a}_{\alpha}}^{L},{\tilde{a}_{\alpha}}^{U}]$.
$\tilde{a}\in F(\mathbb{R})$ can be recovered from its $\alpha$-level sets by a well-known decomposition theorem (ref. \cite{GE1}), which states that $\tilde{a}= \cup_{\alpha \in [0,1]} \alpha \cdot \tilde{a}_{\alpha} $ where union on the right-hand side is the standard fuzzy union.
\end{definition}

\begin{definition}\label{def2}\cite{SO} Let $\tilde{a}, \tilde{b} \in F(\mathbb{R})$ with $\tilde{a}_{\alpha}=[{\tilde{a}_{\alpha}}^{L},{\tilde{a}_{\alpha}}^{U}]$ and  $\tilde{b}_{\alpha}=[{\tilde{b}_{\alpha}}^{L},{\tilde{b}_{\alpha}}^{U}]$. According to Zadeh's extension principle, addition and scalar multiplication in the set of fuzzy numbers $F(\mathbb{R})$ by their $\alpha$-level sets are given as follows:
\begin{eqnarray*}
(\tilde{a}\oplus \tilde{b})_{\alpha} & = &
[\tilde{a}_{\alpha}^{L}+\tilde{b}_{\alpha}^{L}, \tilde{a}_{\alpha}^{U}+\tilde{b}_{\alpha}^{U}] \\
(\lambda \odot \tilde{a})_{\alpha} & = &
[\lambda\cdot\tilde{a}_{\alpha}^{L},\lambda\cdot\tilde{a}_{\alpha}^{U}],~if~\lambda \geq 0 \\
			           & = &
[\lambda\cdot\tilde{a}_{\alpha}^{U},\lambda\cdot\tilde{a}_{\alpha}^{L}],~if~\lambda < 0
\end{eqnarray*}
where $\lambda \in \mathbb{R}$ and $\alpha \in [0,1]$. $(\tilde{a} \oplus \tilde{b})$ and $(\lambda \odot \tilde{a})$ can be determined using the Decomposition theorem stated in Definition \ref{def1}.
\end{definition}
The difference between two fuzzy numbers are defined as follows:

\begin{definition}
For fuzzy numbers $\tilde{a}, \tilde{b}$, with $\alpha$-level sets $\tilde{a}_{\alpha}=[{\tilde{a}_{\alpha}}^{L},{\tilde{a}_{\alpha}}^{U}]$ and  $\tilde{b}_{\alpha}=[{\tilde{b}_{\alpha}}^{L},{\tilde{b}_{\alpha}}^{U}]$,  a difference $(\tilde{a} \ominus \tilde{b})$ is defined using its $\alpha$-level sets as 
\begin{eqnarray*}
(\tilde{a}\ominus \tilde{b})_{\alpha} = [\tilde{a}_{\alpha}^{L} - \tilde{b}_{\alpha}^{U}, \tilde{a}_{\alpha}^{U} - \tilde{b}_{\alpha}^{L}],
\end{eqnarray*}
for all $\alpha \in [0,1]$. Then the difference $(\tilde{a} \ominus \tilde{b})$ is determined using the Decomposition theorem. 
\end{definition}
\begin{definition}
The membership function of a trapezoidal fuzzy number $\tilde{a}$ is defined as
\[
{\tilde{a}}(r) = \left\{
\begin{array}{ll}
 {\frac {(r-a_{1})}{(a_{2} - a_{1})}}~~~ if ~~ a_{1} \leq r \leq a_{2}  \\
 1~~~~~~~~~~if~~ a_{2} \leq r \leq a_{3}\\
 {\frac {(a_{4}-r)}{(a_{4}-a_{3})}}~~~ if ~~ a_{3} < r \leq a_{4} \\
 0~~~~~~~~~~~~~otherwise
\end{array} \right.
\]
It is denoted by $\tilde{a} = (a_{1}, a_{2}, a_{3}, a_{4})$ and has $\alpha$-level set 
$\tilde{a}_{\alpha} = [a_{1} + \alpha(a_{2} - a_{1}), a_{4} - \alpha(a_{4} - a_{3})]$.
\end{definition}

\begin{definition} \cite{SH} \label{def3}
The membership function of a triangular fuzzy number $\tilde{a}$ is defined as
\[
{\tilde{a}}(r) = \left\{
\begin{array}{ll}
 {\frac {(r-a_{1})}{(a- a_{1})}}~~~ if ~~ a_{1} \leq r \leq a  \\
 {\frac {(a_{2}-r)}{(a_{2}-a)}}~~~ if ~~ a < r \leq a_{2} \\
 0~~~~~~~~~~~~~otherwise
\end{array} \right.
\]
which is denoted by $\tilde{a} = (a_{1},a,a_{2})$. The $\alpha$-level set of $\tilde{a}$ is then 
\[
\tilde{a}_{\alpha} = [(1-\alpha)a_{1}+\alpha a, (1-\alpha)a_{2}+\alpha a].
\]
\end{definition}

\begin{definition}\label{def4}\cite{HS2} Let $A, B \subseteq \mathbb{R}^{n}$. The Hausdorff metric $d_H $ is defined by
\begin{eqnarray*}
d_H(A,B) \mathrel{\mathop:}= \max\{\sup_{x \in A}\inf_{y \in B}\|x-y\|, \sup_{y \in B}\inf_{x \in A}\|x-y\|\}. \end{eqnarray*}
Then the \textbf{metric $d_{F}$} on $F(\mathbb{R})$ is defined as 
\begin{eqnarray*}
d_{F}(\tilde{a}, \tilde{b})  \mathrel{\mathop:}=  \sup_{0 \leq \alpha \leq 1}\{d_H({\tilde{a}}_{\alpha},{\tilde{b}}_{\alpha})\},
\end{eqnarray*}
for all $\tilde{a}, \tilde{b} \in F(\mathbb{R})$. Since $\tilde{a}_{\alpha}$ and $\tilde{b}_{\alpha}$ are compact intervals in $\mathbb{R}$,
\begin{eqnarray*}
d_{F}(\tilde{a}, \tilde{b})  =  \sup_{0 \leq \alpha \leq 1} \max\{
|{{\tilde{a}}_{\alpha}}^{L}-{{\tilde{b}}_{\alpha}}^{L}|,
|{{\tilde{a}}_{\alpha}}^{U}-{{\tilde{b}}_{\alpha}}^{U}|\}.
\end{eqnarray*}
Then $(F(\mathbb{R}), d_{F})$ is a complete metric space, see \cite{GE}, with the properties
\begin{itemize}
\item [1.] $d_{F}(\tilde{a}\oplus \tilde{c}, \tilde{b} \oplus \tilde{c}) = d_{F}(\tilde{a}, \tilde{b})$, for all $\tilde{a}, \tilde{b}, \tilde{c} \in F(\mathbb{R})$
\item [2.] $d_{F}(\lambda \odot \tilde{a}, \lambda \odot \tilde{b}) = |{\lambda}| d_{F}(\tilde{a}, \tilde{b})$, for all $\tilde{a}, \tilde{b} \in F(\mathbb{R})$ and $\lambda \in \mathbb{R}$
\item [3.] $d_{F}(\tilde{a}\oplus \tilde{b}, \tilde{c} \oplus \tilde{d}) \leq d_{F}(\tilde{a}, \tilde{c}) + d_{F}(\tilde{b}, \tilde{d})$, for all $\tilde{a}, \tilde{b}, \tilde{c}, \tilde{d} \in F(\mathbb{R})$.
\end{itemize}
\end{definition}
The following results and concepts are known.
\begin{theorem} \label{theorem2.1}(See \cite{BE}) 
\begin{description}
\item [(i)] If we denote $\tilde{0} = \chi_{\{0\}}$ then $\tilde{0} \in F(\mathbb{R})$ is neutral element with respect to $\oplus$,
i.e. $\tilde{a} \oplus \tilde{0} = \tilde{0} \oplus \tilde{a} = \tilde{a}$, for all $\tilde{a} \in F(\mathbb{R})$.
\item [(ii)] With respect to $\tilde{0}$, none of $\tilde{a} \in F(\mathbb{R}) - \mathbb{R}$ , has inverse in $F(\mathbb{R})$ (with respect to $\oplus$).
\item [(iii)] For any $x, y \in \mathbb{R}$ with $x, y \geq 0$ or $x, y \leq 0$ and any $\tilde{a} \in F(\mathbb{R})$ , we have \\ $(x + y) \odot \tilde{a} = x \odot \tilde{a} \oplus y \odot \tilde{a}$; For general $x, y \in \mathbb{R}$ , the above property does not hold.
\item [(iv)] For any $\lambda \in \mathbb{R}$ and any $\tilde{a}, \tilde{b} \in F(\mathbb{R})$ , we have $\lambda \odot (\tilde{a} \oplus \tilde{b}) = \lambda \odot \tilde{a} \oplus \lambda \odot \tilde{b}$;
\item [(v)] For any $\lambda, \mu \in \mathbb{R}$ and any $\tilde{a} \in F(\mathbb{R})$ , we have $\lambda \odot (\mu \odot \tilde{a}) = (\lambda \mu) \odot \tilde{a}$.
\end{description}
\end{theorem}
%=====================================================================================================================
\section{Hukuhara differentiability}
\begin{definition}\cite{HS1}  \label{def5}
Let $V$ be a real vector space and $F(\mathbb{R})$ be a set of fuzzy numbers. Then a function $\tilde{f}: V \to F(\mathbb{R})$ is called fuzzy-valued function on $V$. Corresponding to such a function $\tilde{f}$ and $\alpha \in [0,1]$, we define two real-valued functions $\tilde{f}_{\alpha}^L$ and $\tilde{f }_{\alpha}^U$ on $V$ as ${\tilde{f}_{\alpha}}^{L}(x) = (\tilde{f}(x))_{\alpha}^{L}$ and ${\tilde{f}_{\alpha}}^{U}(x) = (\tilde{f}(x))_{\alpha}^{U}$ for all $x \in V$. These functions $\tilde{f}_{\alpha}^L(x)$ and $\tilde{f }_{\alpha}^U(x)$ are called $\alpha$-level functions of the fuzzy-valued function $\tilde{f}$.
\end{definition}

\begin{remark}
The difference ($\ominus$) between two fuzzy numbers is a usual fuzzy difference based on interval arithmetic. This difference always exists. But the important property is not valid, i.e. $((\tilde{a}+\tilde{b}) \ominus \tilde{b})_{\alpha} \neq \tilde{a}_{\alpha}$. Even, $\tilde{a} \ominus \tilde{a} \neq \tilde{0}$, where $\tilde{0} = \chi_{\{0\}}$.
\end{remark}

To overcome this situation, Hukuhara difference is defined.

\begin{definition}
Let $\tilde{a}$ and $\tilde{b}$ be two fuzzy numbers. If there exists a fuzzy number $\tilde{c}$  such that $\tilde{c} \oplus \tilde{b}= \tilde{a}$. Then $\tilde{c}$ is called Hukuhara difference of $\tilde{a}$ and $\tilde{b}$ and is denoted by $\tilde{a} \ominus_{H}\tilde{b}$.
\end{definition}

We have following remarks:

\begin{remark}
\begin{itemize}
\item [1.] The necessary condition for Hukuhara difference $\tilde{a} \ominus_{H}\tilde{b}$ to exist is that some translate of $\tilde{b}$ is a fuzzy subset of $\tilde{a}$. For instance,  $\tilde{a}= (-1, 1, 3)$ and $\tilde{b}= (-1, 0, 1)$ are triangular fuzzy numbers such that $(-1, 1, 3) = (0, 1, 2) \oplus (-1, 0, 1)$ then $(0,1,2)$ is called H-difference of $\tilde{a}$ and $\tilde{b}$, as $\tilde{b}$ is a fuzzy subset of $\tilde{a}$.
\item [2.] Here $(\tilde{a} \odot_{H} \tilde{a}) = \tilde{0}$ and $(\tilde{a}+\tilde{b}) \odot_{H} \tilde{b} = \tilde{a}$, for all $\tilde{a}, \tilde{b} \in F(\mathbb{R})$.
\end{itemize}
\end{remark}

Hukuhara differentiability ( H-differentiability ) of a fuzzy-valued function due to Puri and Ralescu \cite{PU} define as follows:

\begin{definition}\label{def6} Let $X$ be a subset of $\mathbb{R}$. A fuzzy-valued function
$\tilde{f}: X \to F(\mathbb{R})$ is said to be H-differentiable at
$x^{0} \in X$ if and only if there exists a fuzzy number $D\tilde{f}(x^{0})$ such that the limits (with respect to metric $d_{F}$)
\[ \lim_{h \rightarrow 0^{+}} {{\frac{1}{h}}\odot [{\tilde{f}(x^{0}+h)\ominus_{H}\tilde{f}(x^{0})}]},~~and~~
 \lim_{h \rightarrow 0^{+}} {{\frac{1}{h}}\odot [{\tilde{f}(x^{0})\ominus_{H}\tilde{f}(x^{0}-h)}]}\]
 both exist and are equal to ${D\tilde{f}}(x^{0})$. In this case,
 $D\tilde{f}(x^{0})$ is called the H-derivative of $\tilde{f}$ at $x^{0}$. If $\tilde{f}$ is H-differentiable at any $x \in X $, we call $\tilde{f}$ is H-differentiable over $X$.
\end{definition}

We have following proposition regarding differentiability of $\tilde{f}_{\alpha}^{L}$ and $\tilde{f}_{\alpha}^{U}$.

\begin{proposition}\label{proh}
Let $X$ be a subset of $\mathbb{R}$. If a fuzzy-valued function $\tilde{f}:X \to F(\mathbb{R})$ is H-differentiable at $x^{0}$ with derivative ${D\tilde{f}}(x^{0})$, then ${\tilde{f}_{\alpha}^{L}}(x)$ and ${\tilde{f}_{\alpha}^{U}}(x)$ are differentiable at $x^{0}$, for all $\alpha \in [0,1]$. Moreover, we have
$(D{\tilde{f})_{\alpha}(x^{0})}=
[D({\tilde{f}_{\alpha}^{L})(x^{0})},D({\tilde{f}_{\alpha}^{U})(x^{0})}]$.
\end{proposition}

Bede and Gal have discussed the conditions for existence of H-differentiability of fuzzy-valued functions in \cite{BE}. We put these conditions in the form of following proposition. 

\begin{proposition}\label{pro}
Let $\tilde{c} \in F(\mathbb{R}) ~and~ g: (a,b) \to \mathbb{R}_{+}$ be differentiable at $x_{0} \in (a,b) \subset \mathbb{R}_{+}$. Define $\tilde{f}: (a,b) \to F(\mathbb{R})$ by $\tilde{f}(x) = \tilde{c}\odot g(x)$, for all $x \in (a,b)$. If we suppose that, $g^{\prime}(x_{0}) > 0 $ then Hukuhara differences of $\tilde{f}$ exist and $\tilde{f}$ is H-differentiable at $x_{0}$ with $\tilde{f}^{\prime}(x) = \tilde{c}\odot g^{\prime}(x)$. 
\end{proposition}
\begin{proof}
Since $g^{\prime}(x_{0}) > 0 $, then 
\[ g^{\prime}(x_{0}) = \lim_{h \rightarrow 0} {{g(x_{0} + h) - g(x_{0})} \over {h}}, \]
it follows that for $h > 0$ sufficiently small, we have ${g(x_{0} + h) - g(x_{0})} = \omega(x_{0},h) > 0$. Multiplying by $\tilde{c}$ it follows $\tilde{c} \odot g(x_{0} + h) = \tilde{c} \odot g(x_{0}) \oplus \tilde{c} \odot \omega(x_{0},h)$  (An operator $\odot$ defines multiplication of a fuzzy number with a real number).  That is, there exists Hukuhara difference ${\tilde{f}(x_{0} + h) {\ominus}_{H} \tilde{f}(x_{0})}$. Similarly, by \[ g^{\prime}(x_{0}) = \lim_{h \rightarrow 0} {{g(x_{0}) -g(x_{0} - h) } \over {h}}, \]
reasoning as above, there exists the Hukuhara difference ${\tilde{f}(x_{0}) {\ominus}_{H} \tilde{f}(x_{0}- h)}$. \\
Now, if we suppose $g^{\prime}(x_{0}) < 0 $, we see that we cannot get the above kind of reasoning to prove that the H-differences ${\tilde{f}(x_{0} + h) {\ominus}_{H} \tilde{f}(x_{0})}$, ${\tilde{f}(x_{0}) {\ominus}_{H} \tilde{f}(x_{0}- h)}$ and the derivative $\tilde{f}^{\prime}(x_{0})$ exist. \qed
\end{proof}

The generalized Hukuhara difference ( gH-difference) coming from \cite{ST1,ST2} is more general than Hukuhara difference (H-difference). From an algebraic point of view, the difference between two sets may be interpreted both in terms of addition as in $\ominus_{H}$ or in terms of negative addition. Based on this fact, the generalized Hukuhara difference (gH-difference) is defined as follows:

\begin{definition}\cite{ST1,ST2}\label{def7}
Given two fuzzy numbers $\tilde{a}, \tilde{b} \in F(\mathbb{R})$, the gH-difference is the fuzzy number $\tilde{c}$, if exists, such that 
\begin{equation}
\tilde{a} \ominus_{gH} \tilde{b} = \tilde{c} \Leftrightarrow \left\{
\begin{array}{ll}
 ~~~(i) ~\tilde{a} = \tilde{b} + \tilde{c}, \\
 or~ (ii)~ \tilde{b} = \tilde{a} - \tilde{c}. 
\end{array} \right.
\end{equation}
\end{definition}
In terms of $\alpha$-level sets, $[\tilde{a} \ominus_{gH} \tilde{b}]_{\alpha} = [\min\{ \tilde{a}_{\alpha}^{L} - \tilde{b}_{\alpha}^{L}, \tilde{a}_{\alpha}^{U} - \tilde{b}_{\alpha}^{U}\}, \max\{ \tilde{a}_{\alpha}^{L} - \tilde{b}_{\alpha}^{L}, \tilde{a}_{\alpha}^{U} - \tilde{b}_{\alpha}^{U} \}]$ and if H-difference exists, then $\tilde{a} \odot_{H} \tilde{b}= \tilde{a} \odot_{gH} \tilde{b}$; the conditions for the existence of $\tilde{c} = \tilde{a} \odot_{gH} \tilde{b} \in F(\mathbb{R})$ are
\begin{equation}\label{case1}
case(i)~ \left\{
\begin{array}{ll}
 \tilde{c}_{\alpha}^{L} = \tilde{a}_{\alpha}^L-\tilde{b}_{\alpha}^L ~and~ \tilde{c}_{\alpha}^{U} = \tilde{a}_{\alpha}^{U}-\tilde{b}_{\alpha}^{U}  \\
with~\tilde{c}_{\alpha}^{L} ~ increasing,~\tilde{c}_{\alpha}^{U}~decreasing,~ \tilde{c}_{\alpha}^{L} \leq \tilde{c}_{\alpha}^{U},
\end{array} \right.
\end{equation}
for all $\alpha \in [0,1]$.

\begin{equation}\label{case2}
case(ii)~ \left\{
\begin{array}{ll}
 \tilde{c}_{\alpha}^{L} = \tilde{a}_{\alpha}^U-\tilde{b}_{\alpha}^U ~and~ \tilde{c}_{\alpha}^{U} = \tilde{a}_{\alpha}^{L}-\tilde{b}_{\alpha}^{L}  \\
with~\tilde{c}_{\alpha}^{L} ~ increasing,~\tilde{c}_{\alpha}^{U}~decreasing,~ \tilde{c}_{\alpha}^{L} \leq \tilde{c}_{\alpha}^{U},
\end{array} \right.
\end{equation}
for all $\alpha \in [0,1]$.

We discuss some examples.

\begin{example}
Consider two triangular shape fuzzy numbers $\tilde{a} = (3,4,5)$ and $\tilde{b} = (-3,-2,-1)$. The gH-difference $\tilde{a} \odot_{gH} \tilde{b}$ exists as 
\begin{eqnarray*}(\tilde{a} \odot_{gH} \tilde{b})_{\alpha}^{L} = \min\{ \tilde{a}_{\alpha}^{L}-\tilde{b}_{\alpha}^{L}, \tilde{a}_{\alpha}^U-\tilde{b}_{\alpha}^U\} = \min\{ (3+\alpha)-(-3+\alpha), (5-\alpha)-(-1-\alpha)\} = 6\end{eqnarray*}
and  
\begin{eqnarray*}(\tilde{a} \odot_{gH} \tilde{b})_{\alpha}^{U} = \max\{ \tilde{a}_{\alpha}^{L}-\tilde{b}_{\alpha}^{L}, \tilde{a}_{\alpha}^U-\tilde{b}_{\alpha}^U\} = \max\{ (3+\alpha)-(-3+\alpha), (5-\alpha)-(-1-\alpha)\} = 6.\end{eqnarray*}
It satisfies the conditions for the existence in (\ref{case1}) and (\ref{case2}).
\end{example}
In some case, gH-difference of two fuzzy numbers may not exists.

\begin{example}
Consider a triangular shape fuzzy number $\tilde{a} = (0,2,4)$ and a trapezoidal shape fuzzy number $\tilde{b} = (0,1,2,3)$. The gH-difference $\tilde{a} \odot_{gH} \tilde{b}$ does not exist as 
\begin{eqnarray*}
(\tilde{a} \odot_{gH} \tilde{b})_{0}^{L} = \tilde{a}_{0}^{L}-\tilde{b}_{0}^{L} = 0 < (\tilde{a} \odot_{gH} \tilde{b})_{0}^{U} = \tilde{a}_{0}^U-\tilde{b}_{0}^U = 1 
\end{eqnarray*}
and  
\begin{eqnarray*}
(\tilde{a} \odot_{gH} \tilde{b})_{1}^{L} = 1 > (\tilde{a} \odot_{gH} \tilde{b})_{1}^{U} = 0,
\end{eqnarray*}
so neither (\ref{case1}) nor (\ref{case2}) hold for any $\alpha \in [0,1]$.
\end{example}

\begin{par}
 Bede and Stefanini\cite{BE1} have defined Generalized Hukuhara differentiability ( gH-differentiability ) of fuzzy-valued functions based on gH-difference. The generalized Hukuhara differentiability (gH-differentiability ) of the fuzzy-valued functions is defined as:
\end{par}
\begin{definition}\cite{BE1}\label{def8}
Let $x_{0} \in X$, $X$ is an open interval and $h$ be such that $x_{0} + h \in X$, then gH-derivative of a function $\tilde{f}: X \to F(\mathbb{R})$ at $x_{0}$ is defined as
\begin{equation}\label{eq}
\tilde{f}^{\prime}_{gH}(x_{0}) = \lim_{h \to 0} {{{1} \over {h}} \tilde{f}(x_{0} + h ) \ominus_{gH} \tilde{f}(x_{0}). }                                      
\end{equation}
If $\tilde{f}^{\prime}_{gH}(x_{0}) \in F(\mathbb{R})$ satisfying (\ref{eq}) exists, we say that $\tilde{f}$ is generalized Hukuhara differentiable ( gH-differentiable ) at $x_{0}$. If $\tilde{f}$ is gH-differentiable at each $x \in X $ with gH-derivative $\tilde{f}^{\prime}_{gH}(x) \in F(\mathbb{R})$, we say that $\tilde{f}$ is gH-differentiable on $X$. Moreover, we can define continuously gH-differentiable fuzzy-valued function using continuity of gH-differentiable fuzzy-valued function. For continuity of a fuzzy-valued function, refer Definition 3.2 in \cite{PI}.
\end{definition}
Now we consider some fuzzy-valued functions and check existence of Hukuhara and generalized Hukuhara differentiability. First we consider a constant fuzzy-valued function. 
\begin{example}\label{illus1}
Consider a fuzzy-valued function $\tilde{f}(x) = \tilde{a}$, $x \in \mathbb{R}$ and $\tilde{a}$ is a fuzzy number with $\alpha$-level sets $\tilde{a}_{\alpha} = [\tilde{a}^{L}_{\alpha}, \tilde{a}^{U}_{\alpha}]$. We claim that $\tilde{f}$ is H-differentiable. $\tilde{f}(x)$ can be written as  $\tilde{f}(x) = \tilde{a}\odot g(x)$ where $g(x) = 1 $ for $x\in \mathbb{R}$. Since $g^{\prime}(x) = 0$, we can not use Proposition \ref{proh}. For sufficiently small $h> 0$, $g(x+h)-g(x) = 0$. We can write $g(x+h) = g(x) + 0$. Multiplying by fuzzy number $\tilde{a}$, $\tilde{a}\odot g(x+h) = \tilde{a} \odot (g(x) + 0 ) = \tilde{a}\odot g(x) \oplus \tilde{a}\odot 0$ (using Theorem \ref{theorem2.1} as $g(x) = 1 > 0 $). Therefore, there exist H-difference $\tilde{f}(x+h)\ominus_{H} \tilde{f}(x)$ which is zero. Similar way, we show existence of $\tilde{f}(x)\ominus_{H} \tilde{f}(x-h)$. Therefore, Hukuhara derivative of constant fuzzy-valued function $\tilde{f}(x) = \tilde{a}$ exists and its derivative is $\tilde{f}^{\prime}(x) = \tilde{0}$. Here we see that $\alpha$-level functions $\tilde{f}^{L}_{\alpha}(x)= \tilde{a}^{L}_{\alpha}$ and $\tilde{f}^{U}_{\alpha}(x) = \tilde{a}^{U}_{\alpha} = $ are differentiable functions with respect to $x$ for all $\alpha \in [0,1]$. 
\end{example}
Consider a linear fuzzy-valued function. 
\begin{example}
Consider a fuzzy-valued function $\tilde{f}(x) = \tilde{a}\odot x$, $x \in \mathbb{R}$  $( \tilde{a}$ is a fuzzy number). For $x > 0 $, $\tilde{f}(x)$ is Hukuhara differentiable, since $\tilde{f}(x) = \tilde{a}\odot g(x)$ where $g(x): \mathbb{R}_{+} \to \mathbb{R}_{+}$ with $g^{\prime}(x) = 1 > 0$ (see Proposition \ref{proh}). For $x < 0 $, $\alpha$-level functions $\tilde{f}_{\alpha}^{L}(x) = \tilde{a}_{\alpha}^{U}x$ and $\tilde{f}_{\alpha}^{U}(x) = \tilde{a}_{\alpha}^{L}x$ of the fuzzy-valued function $\tilde{f}(x) = \tilde{a}\odot x$. For sufficiently small $h>0$, $\tilde{f}_{\alpha}^{L}(x+h) - \tilde{f}_{\alpha}^{L}(x) = \tilde{a}_{\alpha}^{U} h$ and $\tilde{f}_{\alpha}^{U}(x+h) - \tilde{f}_{\alpha}^{U}(x) = \tilde{a}_{\alpha}^{L} h$. Hukuhara difference does not exist as $\tilde{f}_{\alpha}^{L}(x+h) - \tilde{f}_{\alpha}^{L}(x) = \tilde{a}_{\alpha}^{U} h \nleq \tilde{f}_{\alpha}^{U}(x+h) - \tilde{f}_{\alpha}^{U}(x) = \tilde{a}_{\alpha}^{L} h$ ( $\tilde{a}_{\alpha}^{L} \leq \tilde{a}_{\alpha}^{U} $, $\alpha \in [0,1]$ and $h > 0$). Therefore, $\tilde{f}(x)$ is not Hukuhara differentiable for $x<0$. We also see that $\alpha$-level functions $\tilde{f}_{\alpha}^{L}(x) = \tilde{a}_{\alpha}^{U}x$ and $\tilde{f}_{\alpha}^{U}(x) = \tilde{a}_{\alpha}^{L}x$, $\alpha \in [0,1]$ are differentiable functions of $x$. We see that $\tilde{f}(x)$ is gH-differentiable, for $x < 0$. Because, here Hukuhara differences are given as 
\begin{eqnarray*}
 [\tilde{f}(x+h) \ominus_{gH} \tilde{f}(x)]_{\alpha}& = & [\min\{\tilde{f}_{\alpha}^{L}(x+h) - \tilde{f}_{\alpha}^{L}(x), \tilde{f}_{\alpha}^{U}(x+h) - \tilde{f}_{\alpha}^{U}(x)\}, \\
&  & \max\{\tilde{f}_{\alpha}^{L}(x+h) - \tilde{f}_{\alpha}^{L}(x), \tilde{f}_{\alpha}^{U}(x+h) - \tilde{f}_{\alpha}^{U}(x)\}] \\
& = & [\tilde{a}_{\alpha}^{L} h, \tilde{a}_{\alpha}^{U} h]
\end{eqnarray*}
for all $\alpha \in [0,1]$ and $h > 0$. While gH-differences 
\begin{eqnarray*}
\tilde{f}(x) \ominus_{gH} \tilde{f}(x-h)]_{\alpha} & = & [\min\{\tilde{f}_{\alpha}^{L}(x) - \tilde{f}_{\alpha}^{L}(x-h), \tilde{f}_{\alpha}^{U}(x) - \tilde{f}_{\alpha}^{U}(x-h)\}, \\
&  & \max\{\tilde{f}_{\alpha}^{L}(x) - \tilde{f}_{\alpha}^{L}(x-h), \tilde{f}_{\alpha}^{U}(x) - \tilde{f}_{\alpha}^{U}(x-h)\}] \\
& = & [\tilde{a}_{\alpha}^{L} h, \tilde{a}_{\alpha}^{U} h] 
\end{eqnarray*} for $\alpha \in [0,1]$ and $h > 0$. Now we check Hukuhara differentiability of the fuzzy-valued function on $\mathbb{R}$. We see that $\alpha$-level functions $\tilde{f}_{\alpha}(x)$ are defined as 
\begin{eqnarray*}
\tilde{f}_{\alpha}(x) = \left\{
\begin{array}{ll}
 [\tilde{a}_{\alpha}^{L}x,\tilde{a}_{\alpha}^{U}x] ~~~ when ~~ x \geq 0,  \\
 {[\tilde{a}_{\alpha}^{U}x,\tilde{a}_{\alpha}^{L}x] ~~~ when ~~ x < 0}
\end{array} \right.
\end{eqnarray*}
are not differentiable at $x = 0 $. Therefore, we say that the fuzzy-valued function $\tilde{a}\odot x$, $x \in \mathbb{R}$ is not Hukuhara differentiable at $x = 0$. However, it is gH-differentiable on $\mathbb{R}$ and is gH-derivative is $\tilde{f}^{\prime}(x) = \tilde{a}$.
\end{example}
Consider an another example of quadratic fuzzy-valued function.
\begin{example}
Let $\tilde{f}: [-1,1] \to F(\mathbb{R})$ be a fuzzy-valued function defined as $\tilde{f}(x) = \tilde{a}\odot x^{2}$ where $\tilde{a}$ is a fuzzy number with $\alpha$-level sets $\tilde{a}_{\alpha} = [\tilde{a}^{L}_{\alpha}, \tilde{a}^{U}_{\alpha}]$. Its $\alpha$-level functions are defined as $\tilde{f}_{\alpha}(x) =  [\tilde{a}_{\alpha}^{L}x^2 ,\tilde{a}_{\alpha}^{U}x^2]$ for $x \in [-1,1]$. First we check Hukuhara differentiability of $\tilde{f}$. We express $\tilde{f}$ as $\tilde{f}(x) = \tilde{a} \odot g(x)$, where $g(x) = x^2$, $x \in [-1,1]$. For $ -1 \leq x < 0 $ and sufficiently small $h > 0$, we have $g(x+h) - g(x) = w(x,h)$, $w(x,h) = 2xh+ h^2$. We can not find a fuzzy number of the form $\tilde{a}\odot w(x,h)$ for $x < 0 $. Therefore, in this case Hukuhara difference does not exist and hence given fuzzy-valued function is not H-differentiable. For $ 0 < x \leq 1 $, $g(x) = x^2$ with $g^{\prime}(x) = 2x > 0$. By Proposition \ref{pro}, Hukuhara differences exist for given fuzzy-valued function. Now at $x = 0$, for sufficiently small $h > 0$, we have $g(x+h) - g(x) = w_{1}(h)$, $w_{1}(h) = h^2$. Here we can defined a fuzzy number in the form $\tilde{a} \odot w_{1}(h) = \tilde{a} \odot h^2$. Therefore, H-difference exists. Hence, we observed that the given fuzzy-valued function is Hukuhara differentiable for $ 0 \leq x \leq 1 $ while its $\alpha$-level functions $\tilde{f}_{\alpha}^{L}(x) =  \tilde{a}_{\alpha}^{L}x^2 $ and $\tilde{f}_{\alpha}^{U}(x) =  \tilde{a}_{\alpha}^{U}x^2 $ are differentiable for $x \in [-1,1]$. It is easy to see that the given fuzzy-valued function is gH-differentiable for $x \in [-1,1]$.
\end{example}
\begin{par}
In general, a fuzzy-valued function $\tilde{f}: \mathbb{R} \to F(\mathbb{R})$ with $\tilde{f}(x) = \tilde{a} \odot g(x)$ where $g$ is a real-valued function with domain $x < 0 $ or $g^{\prime}(x) < 0 $ is not Hukuhara differentiable but it is gH-differentiable.  In both the cases, its lower and upper level functions may or may not be differentiable. 
\end{par}
We state one important property of H-differentiability.
\begin{theorem}\cite{PH}
If $\tilde{f}, \tilde{g}: X \to F(\mathbb{R})$ are Hukuhara differentiable at $x_{0} \in X \subseteq \mathbb{R}$, then $\tilde{f} \oplus \tilde{g}$ and $c \odot \tilde{f}$, for $c \in \mathbb{R}$ are Hukuhara differentiable at $x_{0}$ and 
\begin{eqnarray*}
(\tilde{f} \oplus \tilde{g})^{\prime}(x_{0}) = \tilde{f}^{\prime}(x_{0}) \oplus \tilde{g}^{\prime}(x_{0}),~(c\odot\tilde{f})^{\prime}(x_{0}) = c\odot \tilde{f}^{\prime}(x_{0})
\end{eqnarray*}
\end{theorem}
\begin{remark}
By this property, we have a class of H-differentiable fuzzy-valued functions which are linear combination of H-differentiable fuzzy-valued functions.
\end{remark}
\begin{par}
We define second order and higher order H-differentiability under the assumption that the fuzzy-valued function is H-differentiable. $C^{n}(X, F(\mathbb{R}))$, $n \geq 1$ denotes the space of n-times continuously H-differentiable fuzzy-valued functions from $X \subseteq \mathbb{R}$ to $F(\mathbb{R})$. By Theorem 5.2 of \cite{KA} for $\tilde{f} \in C^{n}(X, F(\mathbb{R}))$ we have
\begin{eqnarray*}
(\tilde{f}^{(i)}(x))_{\alpha} = [(\tilde{f}_{\alpha}^{L}(x))^{(i)}, (\tilde{f}_{\alpha}^{U}(x))^{(i)}],
\end{eqnarray*}
for $i = 0,1,2,...n$ and $\alpha \in [0,1]$. Here $i$ denotes $i^{th}$ order derivative.
\end{par}
We discuss higher order H-differentiability by some illustrations.
\begin{example}
Let $\tilde{f}: [0,\pi] \to \mathbb{R}$ be a fuzzy-valued function defined as $\tilde{f}(x) = \tilde{a}\odot \sin(x)$ where $\tilde{a}$ is a fuzzy number with $\alpha$-level sets $\tilde{a}_{\alpha} = [\tilde{a}^{L}_{\alpha}, \tilde{a}^{U}_{\alpha}]$. We see that $g(x) = \sin(x)$ is non-negative for $x \in [0,\pi]$ but $g^{\prime}(x) = \cos(x) \geq 0$ for $x \in [0, \pi/2]$. Therefore, H-differences exist for $x \in [0, \pi/2]$. Hence, the function $\tilde{f}(x) = \tilde{a}\odot \sin(x)$ is H-differentiable on $x \in [0, \pi/2]$ while it is gH-differentiable for $x \in [0,\pi]$. Here we see that the $\alpha$-level functions $\tilde{f}_{\alpha}^{L}(x) =  \tilde{a}_{\alpha}^{L}\sin(x) $ and $\tilde{f}_{\alpha}^{U}(x) =  \tilde{a}_{\alpha}^{U}\sin(x)$ are differentiable for $x \in [0,\pi]$. Now to check the second order H-differentiability of the given fuzzy-valued function, we consider $\tilde{f}(x) = \tilde{a}\odot \sin(x)$ for $x \in [0,\pi/2]$ where it is H-differentiable with H-derivative $\tilde{f}^{\prime}(x) = \tilde{a}\odot \cos(x)$. We see that the level functions $\tilde{a}_{\alpha}^{L}\cos(x)$ and $\tilde{a}_{\alpha}^{U}\cos(x)$ are differentiable but it is not H-differentiable because here $g^{\prime}(x) = -\sin(x) \leq 0$ for $x \in [0,\pi/2]$. It is to be noticed that the function is twice gH-differentiable on $x \in [0,\pi]$.
\end{example}
%====================================================================================================================

\begin{example}
Consider a fuzzy-valued function $\tilde{f}(x) = \tilde{a}\odot x^{n}$, $n \geq 1$ defined on $x \in \mathbb{R}_{+}$. Here $g(x) = x^{n} > 0$ with $g^{\prime}(x) = n x^{n-1} > 0$, therefore by Proposition \ref{pro}, H-differences exist and function is H-differentiable on $\mathbb{R}_{+}$ with H-derivative $\tilde{f}^{\prime}(x) = \tilde{a}\odot (nx^{n-1})$. It is easy to see that the function $\tilde{f}^{\prime}(x)$ is H-differentiable as $g(x) = (nx^{n-1}) > 0 $ with $g^{\prime}(x) = n(n-1) x^{n-2} > 0 $ for $x \in \mathbb{R}_{+}$. This function is n-times H-differentiable on $x \in \mathbb{R}_{+}$.
\end{example}
\begin{example}
Consider a fuzzy-valued function $\tilde{f}(x) = \tilde{a}_{n}\odot x^{n} \oplus \tilde{a}_{n-1}\odot x^{n-1} \oplus...\oplus \tilde{a}_{1}\odot x \oplus \tilde{a}_{0} $, $n \geq 0$ defined on $x \in \mathbb{R}_{+}$. It is easy to see that the function is n-times H-differentiable on $x \in \mathbb{R}_{+}$. 
\end{example}
\begin{example}
Consider a fuzzy-valued function $\tilde{f}(x) = \tilde{a}\odot e^{x}$ defined on $x \in I = \mathbb{R}_{+} \cup \{0\}$. Here $g(x) = g^{\prime}(x) = g^{\prime\prime}(x) =....= e^{x} > 0$ for $x \in I$, therefore by Proposition \ref{pro}, the function is n-times H-differentiable on $x \in I$.
\end{example}
%=====================================================================================================================
We turn to Hukuhara differentiability of multi-variable fuzzy-valued functions. Proposition \ref{pro} can be generalized for multi-variable fuzzy-valued function.
\begin{proposition}\label{pro1}
Let $\tilde{c} \in F(\mathbb{R}) ~and~ g: I^{n} \to \mathbb{R}_{+}$, $ ( I^{n} = I \times ... \times I ~(n ~times)$, $I = (a,b) \subset \mathbb{R}_{+} )$  be differentiable at $\bar{x}_{0} \in I^{n}$. Define $\tilde{f}: I^{n} \to F(\mathbb{R})$ by $\tilde{f}(\bar{x}) = \tilde{c}\odot g(\bar{x})$, for all $\bar{x} \in I^{n}$. If we suppose that the partial derivatives $ {{\partial g(\bar{x}_{0})} \over {\partial x_{i}}}  > 0 $, $i = 1,...,n$ then Hukuhara differences $ ({\tilde{f}(\bar{x}_{0} + h) {\ominus}_{H} \tilde{f}(\bar{x}_{0})}$, ${\tilde{f}(\bar{x}_{0}) {\ominus}_{H} \tilde{f}(\bar{x}_{0}- h)}$ for sufficiently small $ h > 0 ) $  exist. Therefore, the partial H-derivatives of $\tilde{f}$ exist and is given as
$ {{\partial \tilde{f}} \over {\partial x_{i}}} = \tilde{c} \odot {{\partial g} \over {\partial x_{i}}}$,$i = 1,...,n$ at $\bar{x}_{0}$. 
\end{proposition}

\begin{proof}
The partial derivative $ {{\partial g(\bar{x}_{0})} \over {\partial x_{i}}}  > 0 $, $i = 1,...,n$ is an ordinary derivative with respect to each $x_{i}$. Using the same reasoning used in the Proposition \ref{pro}, we have the proof. \qed
\end{proof}

\begin{definition}\label{def10}
We say that $\tilde{f}$ is H-differentiable at $\bar{x}^{0}$ if and only if one of the partial H-derivatives ${\partial\tilde{f}}/ {\partial x_{1}},..., {\partial\tilde{f}}/ {\partial x_{n}}$ exists at $\bar{x}^{0}$ and the remaining n-1 partial H-derivatives exist on some neighbourhood of $\bar{x}^{0}$ and are continuous at $\bar{x}^{0}$ (in the sense of fuzzy-valued function).
\end{definition}
\begin{remark}
To check second order Hukuhara differentiability of fuzzy-valued function $\tilde{f}(\bar{x}) = \tilde{c}\odot g(\bar{x})$ , we have to check second order derivatives $ {{\partial^2 g(\bar{x}_{0})} \over {{\partial x_{i}}{\partial x_{j}}}}  > 0 $, $i,j = 1,...,n$.
\end{remark}
%=====================================================================================================================
\section{Fuzzy modeling and H-differentiability}
Fuzzy modeling means develop a mathematical model of real world problem under uncertain situations. If the parameters in the system are known approximately, we model the system as a fuzzy system. We discuss some aspects of fuzzy modeling. Consider a crisp function 

\begin{equation}\label{c1}
f(x) = 2x,
\end{equation} 
where $x \in \mathbb{R}_{+}$. If we fuzzify the function by approximating real number $2$ or making it to fuzzy, we have a fuzzy-valued function 

\begin{equation}\label{f2}
\tilde{f}_{1}(x) = \tilde{2} \odot x,
\end{equation} 
where $\tilde{2}$ is a fuzzy number. Specifically, we take $\tilde{2} = (0,2,4)$ a triangular fuzzy number. We calculate $\tilde{f}_{1}(1) = \tilde{2} \odot 1 = (0,2,4)$, $\tilde{f}_{1}(2) = \tilde{2} \odot 2 = (0,4,8)$,...and so on. The $\alpha$-level sets of $\tilde{f}_{1}(x)$ are $[(0 + 2\alpha)x, (4-2\alpha)x]$, for $x\in \mathbb{R}_{+}$ and $\alpha \in [0,1]$. The $\alpha$-level functions are $(f_{1})_{\alpha}^{L}(x) = (2\alpha)x$ and $(f_{1})_{\alpha}^{U}(x) = (4-2\alpha)x$. The graphical representation of $\alpha$-level functions for $x = 1,2,3,4,5$ are shown in Fig. 1. In Fig. $1(a)$, the graphs of lower level functions $(f_{1})_{\alpha}^{L}(x) = (2\alpha)x$ are drawn with respect to $\alpha$ by fixing $x = 1,2,3,4,5$. In Fig. $1(b)$, the graphs of upper level functions $(f_{1})_{\alpha}^{U}(x) = (4-2\alpha)x$ are drawn with respect to $\alpha$ by fixing $x = 1,2,3,4,5$. 
\begin{figure}[h]\label{fig1}  
\centering
\includegraphics[width=5.0in]{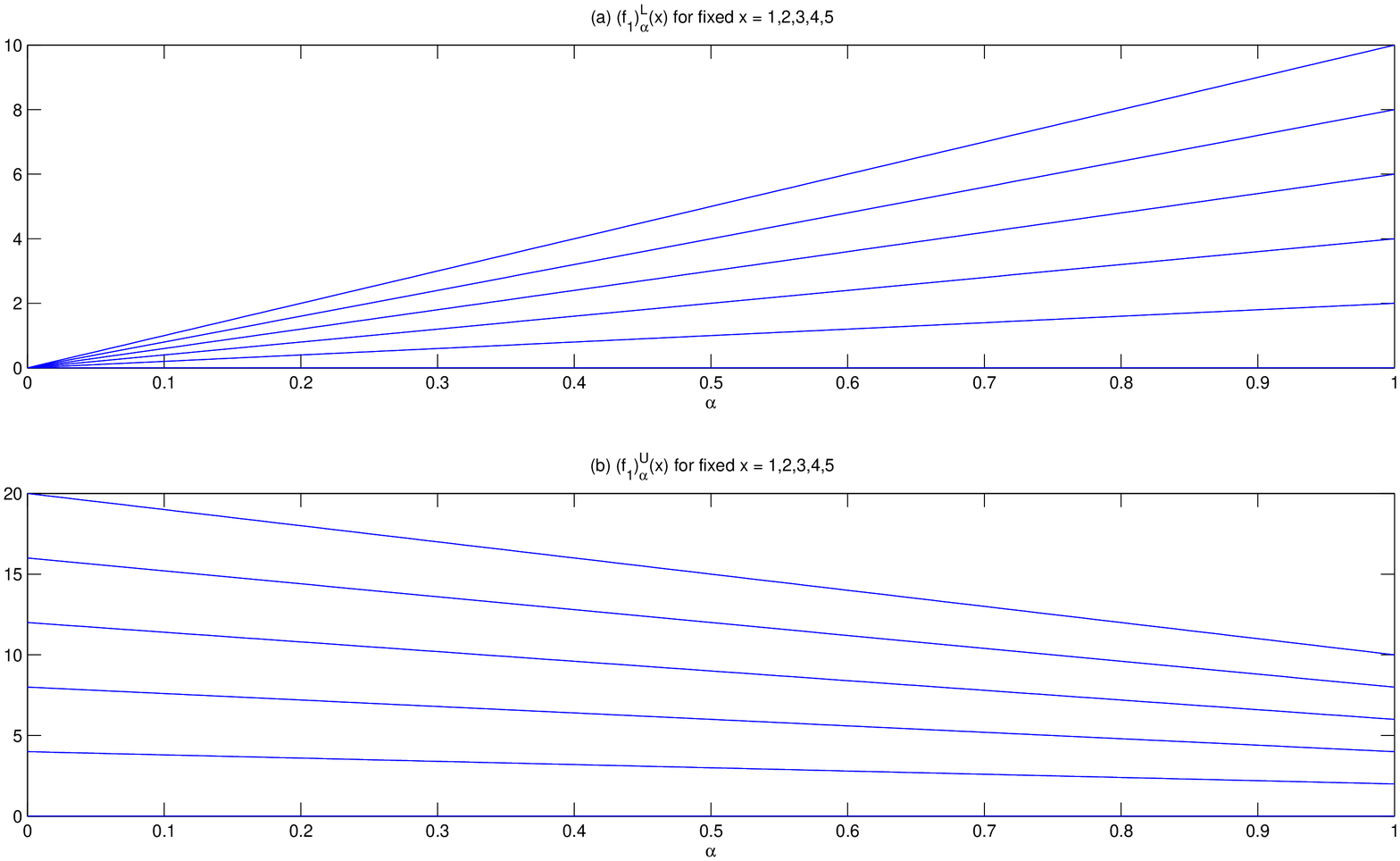}
\caption{}
\end{figure}

Now if we fuzzify the crisp function as 
\begin{equation}\label{f3}
\tilde{f}_{2}(x) = \tilde{1} \odot (2x),
\end{equation} 
where $\tilde{1} = (0,1,2)$ then we have $\tilde{f}_{2}(1) = \tilde{1} \odot 2 = (0,2,4)$, $\tilde{f}_{2}(2) = \tilde{1} \odot 2 \cdot 2 = (0,4,8)$. The $\alpha$-level sets of $\tilde{f}_{2}(x)$ are $[(0 + \alpha)(2x), (2-\alpha)(2x)]$, for $x\in \mathbb{R}_{+}$. The graphical representation of $\alpha$-level functions for $x = 1,2,3,4,5$ are shown in Fig. 2. The Fig. 2(a) shows the graphs of lower level functions $(f_{2})_{\alpha}^{L}(x) = \alpha(2x)$ and Fig. 2(b) shows the graphs of upper level functions $(f_{2})_{\alpha}^{U}(x) = (2-\alpha)2x$, with respect to $\alpha \in [0,1]$, for fixed values of $x = 1,2,3,4,5$.
\begin{figure}[h]\label{fig2}
\centering
\includegraphics[width=5.0in]{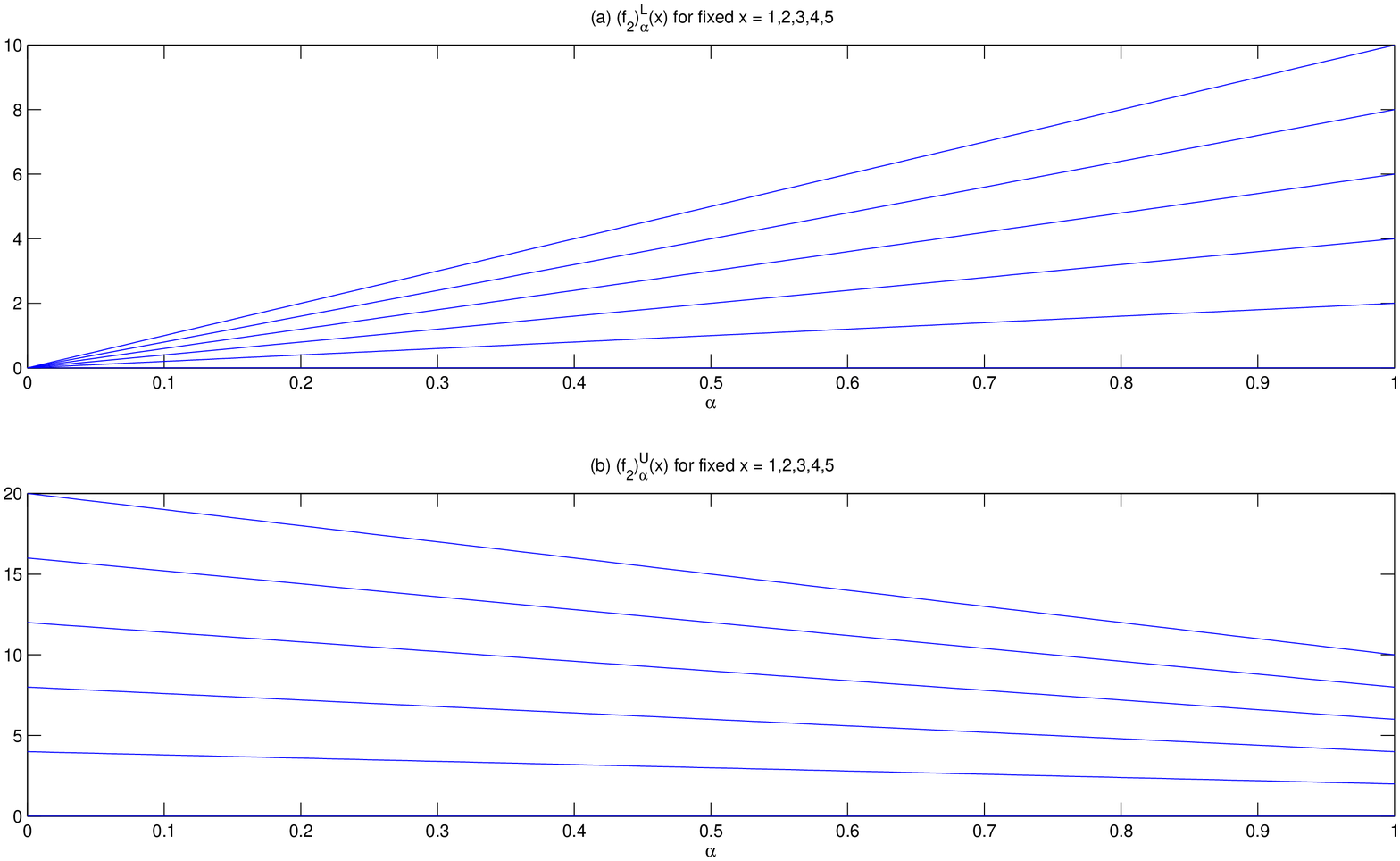}
\caption{}
\end{figure}
\begin{par}
We see that we can make a fuzzy model by different ways of fuzzification which sometimes gives same fuzzy output and sometimes it may gives different output. For instance, we fuzzify the crisp function (\ref{c1}) using (\ref{f2}), by taking $\tilde{2} = (1,2,3) $. Here we have changed the spread of fuzzy coefficient $\tilde{2}$. Then $\tilde{f}(1) = \tilde{2} \odot 1 = (1,2,3)$ where as (\ref{f3}) gives $\tilde{f}(1) = \tilde{1} \odot 2 = (0,1,2) \odot 2 = (0,2,4)$. We see that though the centre value of both fuzzy-valued outputs are same but the spreads are different. By Fig. 1 and 2, we say that both functions $\tilde{f}_{1}$ and $\tilde{f}_{2}$ have same $\alpha$-level functions. Therefore, in fuzzy modelling we can use function $\tilde{f}_{1}$ or $\tilde{f}_{2}$ as a fuzzy model for given crisp function.
\end{par}
%=====================================================================================================================
\begin{par}Now we see the effect of modelling on existence of H-differentiability. Consider a multi-variable fuzzy-valued function and check its H-differentiability.
\end{par}
\begin{example}\label{mex1}
\[
\tilde{f}(x_{1},x_{2}) = (\tilde{1} \odot x_{1}^{3}) \oplus (\tilde{2} \odot x_{2}^{3}) \oplus (\tilde{1}\odot x_{1}\cdot x_{2}),  
\] 
where $\tilde{1} = (-1,1,3)$ and $\tilde{2} = (1,2,3)$ are triangular fuzzy numbers and $x_{1},x_{2} \in \mathbb{R}_{+}$. 
Using fuzzy arithmetic, the $\alpha$-level functions are

\begin{equation}\label{eq1}
\tilde{f}_{\alpha}^{L}(x_{1},x_{2}) = (-1 + 2\alpha) x_{1}^{3} + (1 + \alpha) x_{2}^{3} + (-1 + 2\alpha) (x_{1}\cdot x_{2}),
\end{equation}

\begin{equation}\label{eq2}
\tilde{f}_{\alpha}^{U}(x_{1},x_{2}) = (3-2\alpha) x_{1}^{3} + (3 - \alpha) x_{2}^{3} + (3-2\alpha) (x_{1}\cdot x_{2}),
\end{equation}

for $\alpha \in [0,1]$. To check existence of Hukuhara differentiability, we put the given fuzzy-valued function in the following form
\[
\tilde{f}(x_{1},x_{2}) = \tilde{1}\odot g_{1}(x_{1},x_{2}) \oplus \tilde{2}\odot g_{2}(x_{1},x_{2}) \oplus \tilde{1}\odot g_{3}(x_{1},x_{2}),
\] 
where $g_{1}(x_{1},x_{2}) = x_{1}^{3} ,g_{2} (x_{1},x_{2})= x_{2}^{3}, g_{3}(x_{1},x_{2})= x_{1}\cdot x_{2}$ are real-valued functions defined on $\mathbb{R}^{2}_{+}$. We see that partial derivative of $g_{1}(x_{1},x_{2}) = x_{1}^{3}$ with respect to $x_{2}$ is zero. So the condition for existence of H-differentiability given in Proposition \ref{pro1} does not hold. Similarly, the partial derivative of $g_{2} (x_{1},x_{2})= x_{2}^{3}$ with respect to $x_{1}$ is zero. We modify the form of fuzzy-valued objective function as:
\[
\tilde{f}(x_{1},x_{2}) = \tilde{1}\odot g_{1}(x_{1},x_{2}) \oplus \tilde{1}\odot g_{2}(x_{1},x_{2}),
\] 
where $g_{1}(x_{1},x_{2}) = ( x_{1}^{3} + 2x_{2}^{3}) , g_{2}(x_{1},x_{2})= x_{1}\cdot x_{2}$. Here we see that the partial derivatives of $g_{1}, g_{2}$ with respect to $x_{1},x_{2}$ are greater than zero, since $x_{1}$ and $x_{2}$ are in $\mathbb{R}_{+}$. Therefore, by Proposition \ref{pro1}, we say that the later fuzzy-valued function is H-differentiable. \\

To check second order H-differentiability of $\tilde{f}$, we need to check $ \frac {\partial ^{2} g(\bar{x}^{0})} {\partial x_{i} \partial x_{j}} > 0 $, for $i,j = 1,2$. But we see that $ \frac {\partial ^{2} g_{1}(\bar{x})} {\partial x_{2} \partial x_{1}} = 0 $ and $ \frac {\partial ^{2} g_{2}(\bar{x})} {\partial x_{1} \partial x_{2}} = 0 $. Therefore, Hukuhara differences for second order H-partial derivatives of $\tilde{f}$ do not exist. We further modify the form of fuzzy-valued objective function as 
\[\tilde{f}(x_{1},x_{2}) = \tilde{1}\odot g(x_{1},x_{2}) ,\]
where $g(x_{1},x_{2}) = x_{1}^3 + 2x_{2}^{3} + x_{1}x_{2}$ and $\tilde{1} = (0,1,2)$. For this function $g(x_{1},x_{2})$, first and second order partial derivatives are positive with respect to both variables. Therefore, the first and second H-partial derivatives exist for the fuzzy-valued function $\tilde{f}(x_{1},x_{2}) = \tilde{1}\odot g(x_{1},x_{2})$. 
\end{example}
\begin{remark}
In general, we can make fuzzy-valued function $\tilde{f}(\bar{x}) = \tilde{c}\odot g(\bar{x})$ , $\bar{x} = (x_{1},x_{2},...,x_{n})$ Hukuhara differentiable by modifying fuzzy modeling. For instance, for $\tilde{f}(x_{1},x_{2}) = \tilde{1}\odot g_{1}(x_{1},x_{2}) \oplus \tilde{2}\odot g_{2}(x_{1},x_{2})$, where $g_{1}(x_{1},x_{2}) = x_{1}$ and $g_{2}(x_{1},x_{2}) = x_{2}$, $x_{1}, x_{2} \in \mathbb{R}_{+}$, $ \frac {\partial g_{1}(\bar{x})} {\partial x_{2}} = 0 $ and $ \frac {\partial g_{2}(\bar{x})} {\partial x_{1}} = 0 $. But if modify the fuzzy model as $\tilde{f}(x_{1},x_{2}) = \tilde{1}\odot g(x_{1},x_{2})$, $g(x_{1},x_{2}) = x_{1} + 2x_{2}$ then $ \frac {\partial g(\bar{x})} {\partial x_{i}} > 0 $, for $i = 1,2$ and hence Hukuhara differences exist. 
\end{remark}
\begin{par}
Differentiability plays key role in solving optimization problems. To obtain the optimality conditions in a fuzzy optimization problem, it is necessary that the fuzzy-valued objective function must be twice differentiable. In order to apply the optimality conditions (see ref. \cite{PA} ) with respect to Hukuhara differentiability of a fuzzy-valued objective function, we need to check the first and second order H-differentiability. Also, to solve a fuzzy optimization problem using Newton method proposed in \cite{PI}), H-differentiability of fuzzy-valued function is needed. So, we consider an optimization problem where parameters of the objective function are fuzzy numbers and check its H-differentiability. We see that by changing the fuzzy model of problem, the fuzzy function can be H-differentiable.
\end{par}
\begin{example}\label{mex2}
The fuzzy model of an optimization problem where the parameters (approximate profits in maximization problem) are fuzzy numbers is given as follows:
\[ \tilde{f}(x_{1}, x_{2}) = \tilde{2}\odot x_{1}^{2} \oplus \tilde{2} \odot x_{2}^2 \oplus \tilde{3},\]
where $\tilde{2} = (1,2,4)$ and $\tilde{3} = (1,3,5)$ are triangular fuzzy numbers and $x_{1}, x_{2} \in \mathbb{R}_{+}$. \\

In order to solve the problem by either \cite{PA} or \cite{PI} method, we need to check the Hukuhara differentiability of the fuzzy-valued function. We express the objective function as
\[ \tilde{f}(x_{1}, x_{2}) = \tilde{2}\odot g_{1}(x_{1},x_{2}) \oplus \tilde{2} \odot g_{2}(x_{1},x_{2}) \oplus \tilde{3} \odot g_{3}(x_{1},x_{2}) ,\]
where $g_{1}(x_{1},x_{2}) = x_{1}^{2}$, $g_{2}(x_{1},x_{2}) = x_{2}^{2}$ and $g_{3}(x_{1},x_{2}) = 1$. Using the same arguments given in previous example, the fuzzy-valued function is not H-differentiable as the partial derivative of $g_{1}(x_{1},x_{2})$ with respect to $x_{2}$ and the partial derivative of $g_{2} (x_{1},x_{2})$ with respect to $x_{1}$ is zero. So the condition for existence of H-differentiability given in Proposition \ref{pro1} does not hold. To make it H-differentiable, we can change the fuzzification of the problem as 
\[
\tilde{f}(x_{1}, x_{2}) = \tilde{1}\odot g(x_{1},x_{2}),
\] 
where $\tilde{1} = (1/2, 1, 2)$ a triangular fuzzy number and $g(x_{1},x_{2}) = 2x_{1}^{2} + 2x_{2}^{2} + 2$. If we do this fuzzy modeling, then fuzzy constant $\tilde{3}$ is changed to $\tilde{2}$. The change is allowable because $3$ belong to fuzzy number $\tilde{2}$ with lower degree of membership. By this new modeling, we can make the fuzzy-valued function $\tilde{f}(x_{1},x_{2})$ H-differentiable and it is possible to study the optimality conditions.
\end{example}
%=====================================================================================================================
\section{Conclusions}
Existence of H-differentiability is discussed in detailed for single-variable and multi-variable fuzzy-valued functions. H-differentiability of polynomial fuzzy-valued functions, sine and exponential fuzzy-valued functions are discussed. We see that some of given functions which are not H-differentiable but gH-differentiable. We also discussed higher order H-differentiability of fuzzy-valued functions. It is very interesting to note that fuzzy modeling effects on the existence of H-differentiability. In Examples \ref{mex1} and \ref{mex2}, we see that if modify the fuzzy model of given functions, they become H-differentiable.

%=====================================================================================================================
\begin{acknowledgements}
This research work is supported by National Board for Higher Mathematics (NBHM) , Department of Atomic Energy (DAE), India.
\end{acknowledgements}
%======================================================================================================================

\end{document}